\documentclass[a4paper, 11pt]{amsart}

\usepackage{enumerate}
\usepackage{mathrsfs}
\usepackage[all]{xy}

\title{Branching laws for discrete Wallach points}
\author{St\'ephane Merigon and Henrik Sepp\"anen}
\thanks{The first author was supported by the fellowship ``Bourse Lavoisier'' 
from the French Foreign Office. The second author was supported by a Post Doctoral Fellowship 
from the Swedish Research Council.}

\date{\today}

\address{St\'ephane Merigon, Fachbereich Mathematik, AG AGF
Technische Universit\"at Darmstadt
Schlo{\ss}gartenstra{\ss}e 7
64289 Darmstadt }
\email{merigon@mathematik.tu-darmstadt.de}

\address{Henrik Sepp\"{a}nen, Fachbereich Mathematik, AG AGF
Technische Universit\"at Darmstadt
Schlo{\ss}gartenstra{\ss}e 7
64289 Darmstadt}
\email{seppaenen@mathematik.tu-darmstadt.de}

\newcommand{\C}{\mathbb{C}}
\newcommand{\R}{\mathbb{R}}
\newcommand{\N}{\mathbb{N}}

\newcommand{\D}{\mathcal{D}}
\newcommand{\uD}{\underline{\mathcal{D}}}

\newcommand{\ch}{\mathcal{H}}
\newcommand{\fa}{\mathfrak{a}}
\newcommand{\fm}{\mathfrak{m}}
\newcommand{\fn}{\mathfrak{n}}
\newcommand{\fg}{\mathfrak{g}}
\newcommand{\fq}{\mathfrak{q}}
\newcommand{\fp}{\mathfrak{p}}
\newcommand{\fk}{\mathfrak{k}}
\newcommand{\fz}{\mathfrak{z}}

\newcommand{\ph}{\varphi}

\newcommand{\ov}{\overline}
\newcommand{\uG}{\underline{G}}

\newcommand{\ti}{\widetilde}
\newcommand{\ra}{\rightarrow}
\newcommand{\mt}{\mapsto}
\newcommand{\nm}[1]{\left|#1\right|}

\newcommand{\spl}[1]{\langle #1 \rangle_{L^2(\pol)}}

\newcommand{\pol}{\partial_l\Omega}
\newcommand{\cpol}{C_0^\infty(\partial_l\Omega)}
\newcommand{\cg}{C_0^\infty(G)^\#}
\newcommand{\ol}{\Omega^{(l)}}
\newcommand{\Pl}{\Phi^{(l)}}
\newcommand{\dxl}{d_*^{(l)}\!x}
\newcommand{\Ol}{\Omega^{(l)}}

\newcommand{\gnlp}{\gamma^{(l)}_{\bsn'}}
\newcommand{\hf}{\frac12}
\newcommand{\dmul}{d\mu_l}

\newcommand{\bsn}{{\boldsymbol\nu}}
\newcommand{\bsl}{{\boldsymbol\lambda}}
\newcommand{\bsr}{{\boldsymbol\rho}}
\newcommand{\bsrl}{{\boldsymbol\rho}^{(l)}}
\newcommand{\Dnp}{\Delta_{\bsn'}}
\newcommand{\Dl}{\Delta_{(l)}}
\newcommand{\Dn}{\Delta_{\bsn}}

\DeclareMathOperator{\tr}{tr}
\DeclareMathOperator{\Aut}{Aut}
\DeclareMathOperator{\Ad}{Ad}
\DeclareMathOperator{\Ind}{Ind}

\DeclareMathOperator{\id}{id}

\DeclareMathOperator{\supp}{supp}

\newtheorem{pro}{Proposition}[section]
\newtheorem{lem}[pro]{Lemma}

\newtheorem{theo}[pro]{Theorem}

\theoremstyle{definition}
\newtheorem{defin}[pro]{Definition}
\newtheorem{rem}[pro]{Remark}

\begin{document}
 
\maketitle

\begin{abstract}
We consider the (projective) representations of the
group of holomorphic automorphisms of a symmetric tube 
domain $V\oplus i\Omega$ that are obtained
by analytic continuation of the holomorphic discrete
series. For a representation corresponding to a discrete
point in the Wallach set, we find the decomposition 
under restriction to the identity component of $GL(\Omega)$. Using Riesz distributions, 
an explicit intertwining operator is constructed as an analytic 
continuation of an integral operator. The density for the
Plancherel measure involves quotients of $\Gamma$-functions and 
the $c$-function for a symmetric cone of smaller rank.
\end{abstract}

\section{Introduction}

Let $\uG$ be the identity component of the group of biholomorphisms of a 
irreducible bounded symmetric domain $\uD$.
The scalar holomorphic discrete series of $\uG$ can be realised in the space of 
holomorphic functions on this domain. By reproducing kernel techniques, 
M. Vergne and H. Rossi \cite{VERO} have shown (see also \cite{BER1,WAL,FAKO}) that it has an analytic continuation as a 
family of (projective) irreducible unitary representations $\pi_\alpha$ of $\uG$, 
parametrised 
by the so-called Wallach set.
Let $r$ be the rank of the domain and $d$ its characteristic number (cf. next section 
for a definition). Then the Wallach set is the union of the half-line 
$\alpha>(r-1)\frac{d}{2}$ and a discrete part consisting of $r$ points 
$l\frac{d}{2}$, $l=0,\ldots, r-1$. 
When $\alpha>p-1$, where $p$ is the genus of $\uD$, the representation spaces are 
weighted Bergman spaces.

Let $\tau$ be an antilinear involution of $\uD$. Then $\D:=\uD^\tau$ is a totally 
geodesic submanifold, hence a Riemannian symmetric space, and $G:=\uG^\tau$ contains 
its group of displacements. Such a domain is called a real bounded symmetric domain. 

When one restricts an irreducible unitary representation of a group to a subgroup, 
the representation need not to be irreducible anymore, and the decomposition into irreducibles is
called a branching law. In our context two branching problems 
have been extensively studied: the 
decomposition of the tensor product representation $\pi_\alpha\widehat\otimes\ov\pi_\alpha$ 
and the restriction of $\pi_\alpha$ to symmetric subgroups 
$G=\uG^\tau$ where $\tau$ is an antilinear involution of $\D$. A formula for the 
first problem and for $\alpha>p-1$ was given without
proof by Berezin for classical domains in \cite{BER2}. H. Upmeier and A. Unterberger 
extended it to all domains and gave a Jordan theoretic proof \cite{UPUN}. 
The second problem was solved (for the same parameters) by G. Zhang and 
(independently) by G. van Dijk and M. Pevzner \cite{ZH3,VDPE}, and also by Y. Neretin 
for classical groups \cite{NE1}. 
Those two problems are in fact similar. The restriction map from $\uD$ to $\D$ (resp. from $\uD\times\uD$ to $\uD$)
gives rise to the Berezin transform on $\D$ (resp. $\uD$), which is a kernel operator. The solution then consists in 
computing the spectral symbol of the Berezin transform, or, if one prefers, in computing the Fourier transform of the 
Berezin kernel. 
In \cite{ZH1} and in \cite[Section 5]{VDPE} the problem of decomposing $\pi_\alpha\widehat\otimes\ov\pi_{\alpha+l}$ where $l\in\N$ is also solved, by the same method. A similar problem is also studied in \cite{FAPE}.

For arbitrary parameters, those problems are more complicated, and no 
general method seems to apply. In \cite{ZHO} the tensor product problem 
for $\uG=SU(2,2)$ is solved for any parameter. In \cite{ZH2} the representation 
$\pi_{\frac{d}{2}}\widehat\otimes\ov\pi_{\frac{d}{2}}$ is decomposed for any 
$\uG$ ($\pi_{\frac{d}{2}}$ is called the minimal representation).
In \cite{NE3}, Y. Neretin solves the restriction problem from $U(r,s)$ to 
$O(r,s)$ ($r\leq s$) for any parameter by analytic continuation of the result for large 
parameters. If $r=s$ the support of the Plancherel formula remains the same for all 
$\alpha>r-1$ (here $d=2$) but when $s-r$ is sufficiently large new pieces appear
when $\alpha$ crosses $p-1=2(r+s)-1$ and the situation gets worse as $\alpha$ approaches 
to $(r-1)\frac{d}{2}$, as he had already explained in \cite{NE2}. For points in the 
discrete Wallach set, the situation is not clear.
In his thesis the second author
manages to decompose the restriction of $SO(2,n)$ to $SO(1,n)$ for any parameter 
\cite{SEP2}, as well as the restriction of the 
minimal representation of $SU(p,q)$ to $SO(p,q)$ \cite{SEP3}, and the minimal 
representation of $Sp(n,\R)$ (resp. $SU(n,n)$) to $GL^+(n,\R)$ (resp. to 
$GL(n,\C)$) \cite{SEP1}.

Assume that $\D$ is of tube type, i.e. that $\D$ is biholomorphic to the tube 
domain $T_\Omega$ over the symmetric cone $\Omega$. Then 
the inverse image of $\Omega$ is a real bounded symmetric domain. In this paper, generalising \cite{SEP1}, we 
establish, for any parameter in the discrete Wallach set,
the branching rule for the restriction of the associated representation of 
$\uG=G_0(T_\Omega)$ to $G=GL_0(\Omega)$. 

We use the model by Rossi and Vergne which realises the representation
given by the $l$-th point in the Wallach set as $L^2(\pol, \mu_l)$, where
$\pol$ is the set of positive semidefinite elements in $\partial \Omega$ 
of rank $l$, and $\mu_l$ is a relatively $G$-invariant measure on $\pol$. 
A key observation is that for any $x$ in $\pol$, the function
$g \mapsto \Delta_\bsn(g^*x)$ on $G$, where $\Delta_\bsn$ is the power function
of the Jordan algebra, transforms like a function in a certain parabolically induced
representation. A naive approach to construct an intertwining operator 
from $L^2(\pol, \mu_l)$ into a direct sum of parabolically 
induced representations would then to weight the functions above by 
compactly supported smooth functions, i.e., to consider mappings
$f \mapsto \int_{\pol}f(x) \Delta_\bsn(g^*x)d\mu_l(x)$, for 
$f$ in $\cpol$. 
It will become clear that this approach is in fact fruitful. However, 
there are two problems that have to be dealt with. First of all, it is
not obvious that the natural target spaces are unitarisable.
Secondly, and more importantly, the integrals above need not converge
for the suitable choice of parameters $\bsn$. However, as we shall see, 
both these problems can be 
solved.

The paper is organised as follows. In Section 2 we recall some facts 
about Jordan algebras and symmetric cones that will be needed in the paper. 
In Section~3 we prove an identity between the restriction of a spherical 
function for the cone $\Omega$ to a cone of lower rank in its boundary and the 
corresponding spherical function for the lower rank 
cone. In Section~4 we define a class of irreducible unitary spherical representations 
that provides target spaces for the integral operators discussed above. These are 
constructed using the Levi decomposition of the group $G$ by twisting parabolically 
induced unitary representations for the semisimple factor of $G$ by a 
certain character. In Section~5 we construct the intertwining operator as an 
analytic continuation of the integral operator above. After this has been taken care of, 
a polar decomposition for the measure $\mu_l$ due to J. Arazy and H. Upmeier \cite{ARUP}
allows to express the restriction 
of the intertwining operator to $K$-invariant vectors in terms of the
Fourier transform for a cone of rank $l$. Using this 
identification, the inversion formula for the Fourier transform
can be used to prove the Plancherel theorem for the branching
problem. In the appendix we provide a framework for certain
restrictions of distributions to submanifolds which will be useful for
giving an analytic continuation for the integral that should give
an intertwining operator. It should be pointed out that the standard
theory for restricting distributions (e.g. \cite[Cor.~8.2.7.]{HOR}) 
does not apply to our situation since the condition
on the wave front set for the distribution is not satisfied. Instead
we have to use restrictions based on extending test
functions in such a way that they are constant in certain
directions from the submanifold (cf. Appendix~\ref{A1}).

We finally  want to mention that branching 
problems related to holomorphic involutions of $\uD$ have also been studied in
\cite{REP,KOB,BSA,PEZH}.

\bigskip

\noindent {\bf Acknowledgement}. The authors would like to thank Karl-Hermann Neeb for enlightening
discussions and comments that led to substantial improvement of the presentation.

\section{Jordan theoretic preliminaries}

Let $V$ be a Euclidean Jordan algebra. It is a commutative real algebra 
with unit element $e$ such that the multiplication operator $L(x)$ satisfies 
$[L(x),L(x^2)]=0$, and provided with a scalar product for which $L(x)$ is 
symmetric. An element is invertible if its quadratic representation 
$P(x)=2L(x)^2-L(x^2)$ is so.
Its cone of invertible squares $\Omega$ is a symmetric cone: it is 
homogeneous under the identity 
component, $G$, of the Lie group 
$GL(\Omega)=\{g\in GL(V) \mid g\Omega=\Omega\}$, and 
it is self-dual. 
It follows that the involution 
$\Theta(g):=g^{-*}:=(g^*)^{-1}$ (where $g^*$ is the adjoint of $g$ with respect 
to the scalar product of $V$) preserves $G$ (which is hence reductive). 
The stabiliser $K=G_e$ of $e$ 
coincides with the identity component of the group $\Aut(V)$ of automorphisms of $V$ and 
with the fixed points of $G$ under the involution $\Theta$, and hence is compact. 
Thus $\Omega$ is a a Riemannian symmetric space.

The tube $T_\Omega=V\oplus i\Omega$ over $\Omega$ in the 
complexification of $V$ is an Hermitian symmetric 
space of the non-compact type, diffeomorphic via the Cayley 
transform to a (tube type) bounded symmetric domain.
Any element of $GL(\Omega)$, when extended complex-linearly, preserves 
$T_\Omega$. In this fashion $G$ is seen as a subgroup of the identity component 
$\uG$ of the group of biholomorphisms of $T_\Omega$.

We assume that $V$ is simple. Then there exists a 
positive integer $r$, called the rank of $V$, such 
that any family of 
mutually orthogonal minimal idempotents has $r$ elements. 
Such a family is called a Jordan frame. Let
$$\fg=\fk\oplus\fp$$ 
be the Cartan decomposition of the Lie algebra $\fg$ of $G$. 
Then the map $x\mt L(x)$ yields an isomorphism $V\ra\fp$.
The subspace generated by a Jordan frame (more precisely by the associated 
multiplication operators) is a maximal abelian subspace of $\fp$ and conversely, 
any maximal abelian subspace of $\fp$ determines (up to order) a 
Jordan frame. From 
now we fix a choice of a Jordan frame $(c_1,\dots,c_r)$ and let
$$\fa=\langle L(c_j),j=1,\dots, r\rangle$$ 
and $A=\exp{\fa}$. 

Any $x$ in $V$ can be written 
\begin{equation*}
x=k\sum_{1\leq j\leq 
r}{\lambda_jc_j}
\end{equation*}
where $k\in K$ and the $\lambda_j$ are real 
numbers, and the family $(\lambda_1,\dots,\lambda_r)$ is
unique up to permutation (its members are called the eigenvalues of $x$).
Then $x$ belongs to $\Omega$ if and only if for all 
$1\leq j\leq r$, $\lambda_j>0$, and this spectral decomposition corresponds
the $KAK$ decomposition of $G$, the $A$-component in the decomposition
being unique up to conjugation by an element of the Weyl group $W=\mathfrak{S}_r$ of $G$. 
The rank of $x$ is defined 
to be the number of its nonzero eigenvalues. 
There exists on $V$ a $K$-invariant polynomial function $\Delta(x)$ (the determinant) and a $K$-invariant linear 
function $\tr(x)$ (the trace) that satisfy
$$\Delta(x)=\prod_{j=1}^r\lambda_j\quad\text{and}\quad\tr(x)=\sum_{j=1}^r\lambda_j.$$
The determinant defines the character
\begin{equation}\label{D:char}
\Delta(g):=\Delta(ge) 
\end{equation}
of the group $G$.

A Jordan frame gives rise to the important Peirce decomposition.
Since multiplications by orthogonal idempotents commute, 
the space $V$ decomposes into a direct sum of
joint eigenspaces for the (symmetric) operators $(L(c_j))_{j=1,\dots, r}$. 
The eigenvalues of $L(c)$ when $c$ is an 
idempotent, belong to $\{0,\hf,1\}$. Let us denote by 
$V(c,\alpha)$ the eigenspace corresponding to the value $\alpha$. The
decomposition into joint eigenspaces is then given by
$$V=\bigoplus_{1\leq i\leq j\leq r}V_{ij},$$
where
$$V_{ii}=V(c_i,1)\cap\bigcap_{j\neq i}V(c_j,0),$$
and when $i\neq j$,
$$V_{ij}=V(c_i,\tfrac12)\cap V(c_j,\tfrac12)\cap\bigcap_{k\not\in
\{i,j\}}V(c_j,0).$$
We have $V_{ii}=\R c_i$ and the $V_{ij}$ all have 
the same dimension $d$, called the degree of the Jordan algebra.

We can now describe the roots of $(\fg,\fa)$. 
Let $(\delta_j)_{j=1,\dots, r}$ be the dual basis of 
$(L(c_j))_{j=1,\dots, r}$ in $\fa^*$. Then the roots  
are
$$\alpha^{\pm}_{ij}=\pm \frac{\delta_j-\delta_i}{2},\quad 
1\leq i<j\leq r,$$
and the corresponding root spaces are
\begin{gather*}
\fg^{+}_{ij}=\{a\square e_i\mid a\in V_{ij}\},\\
\fg^{-}_{ij}=\{a\square e_j\mid a\in V_{ij}\}.
\end{gather*}
where $x\square y=L(L(x)y)+[L(x),L(y)]$.
Let $N$ be the nilpotent subgroup
$$N=\exp{\bigoplus_{1\leq i<j \geq r}\fg^+_{ij}}.$$ 
Then $G$ has the Iwasawa decomposition $G=NAK$.

For any idempotent $c$, the projection on $V(1,c)$ is $P(c)$, and $V(1,c)$ is a 
Jordan subalgebra, hence a Euclidean Jordan algebra with neutral element 
$c$ (note that it is simple with rank the one of $c$). We denote by 
$\Omega_1(c)$ its symmetric cone. In particular for
$$e_j=\sum_{k=1}^l{c_k}.$$
we set 
$$V^{(l)}=V(1,e_l) \quad \text{and}\quad\Omega^{(l)}=\Omega_1(e_l),$$
and also note $G^{(l)}$ the identity component of $G(\Ol)$, $K^{(l)}=G_{e_l}$ and 
$\Delta^{(l)}$ the determinant of $V^{(l)}$.
The principal minors of $V$ are then defined by the formula
$$\Delta_{(j)}(x):=\Delta^{(j)}(P(e_j)(x)).$$
Then $x$ is in $\Omega$ if and only if for all $1\leq j\leq r$, $\Delta_{(j)}(x)>0$.
Let $\bsn\in\C^r$ and set for $x$ in $\Omega$, 
$$\Dn(x)=\Delta_{(1)}^{\nu_1-\nu_2}(x)\Delta_{(2)}^{\nu_2-\nu_3}(x)\dots 
\Delta_{(r-1)}^{\nu_{r-1}-\nu_r}(x)\Delta_{(r)}^{\nu_r}(x).$$
Using the basis $(\delta_j)$, we can identify $\bsn$ with an element of $a_\C^*$.
Then if $a(g)$ is the projection of $g$ on $A$ in the Iwasawa decomposition, 
\begin{equation}\label{E:hcef}
\Dn(gx)=e^{\bsn\log{a(g)}}\Dn(x).
\end{equation}

The action of $G$ on the boundary $\partial \Omega$ of $\Omega$ has $r-1$ 
orbits, which may be parametrised by the rank of its elements. We denote by 
$\pol$ the orbit of rank $l$ elements, i.e.,
$$\pol=Ge_l.$$  
There exists on $\pol$ a unique relatively $G$-invariant 
measure $\mu_l$, which transforms according to
$$d\mu_l(gx)=\Delta^{\frac{ld}{2}}(g)d\mu_l(x).$$
The Hilbert space associated to the Wallach point $l\frac{d}{2}$ is, up to renormalisation, 
isometric to $L^2(\pol,\mu_l)$ \cite[Theorem X.III.4]{FAKO}, and
the representation of $G$ in this picture is then given by 
\begin{equation}
\pi^l(g)f=\Delta^{\frac{ld}{4}}(g)f(g^*\cdot).
\end{equation} 
The measures $\mu_l$ were constructed by M. Lassalle \cite{LAS} and can also be obtain as Riesz 
distribution, thanks to S. Gindikin's theorem \cite[VII.3]{FAKO}.
A major tool for our purpose will be 
the polar decomposition of $\mu_l$ \cite[Theorem 3.2.6]{ARUP}.
Let $\Pi_l=K.e_l$ be the set of idempotents of rank $l$. Then $\pol$ is the disjoint union
$$\pol=\bigsqcup_{u\in\Pi_l}\Omega(u).$$
Since elements of $G$ permute the faces of $\ov{\Omega}$ (which are of the form $\ov{\Omega(u)}$ for idempotents $u$), an action 
is induced on $\Pi_l$, such that the preceding equality defines a $G$-equivariant fibration
$$\pol\ra\Pi_l.$$
For any function $f$ in the space $\cpol$ of smooth functions with compact support on $\pol$,
\begin{equation}\label{E:pd}
\int_{\pol}{fd\mu_l}=\int_{K}{dk\int_{\Ol}{\Dl^{\frac{rd}{2}}(x)f(kx)d_*^{(l)}x}}, 
\end{equation}
where $d_*^{(l)}x$ is the unique $G^{(l)}$-invariant measure on $\Ol$.

The set $\pol$ is not a submanifold of $V$. However
let $V_{\geq l}$ be the (open) set of elements in $V$ with rank bigger or equal than $l$.
To any $l$-element subset $I_l\subset\{1,\dots, r\}$ one can associate the idempotent 
$e_{I_l}=\sum_{j\in I_l}{c_j}$ and the minor 
$\Delta_{I_l}(x):=\Delta(P(e_{I_l})x+e-e_{I_l})$. Then
\begin{equation*}
V_{\geq l}=\bigcup_{I_l\subset\{1,\dots, r\}}\{x\in V\mid \Delta_{I_l}(x)\neq 0\},
\end{equation*}
and \cite[Propositions 3 and 7]{LAS} show that $\pol$ is a (closed) submanifold of $V_{\geq l}$.

\section{An identity between spherical functions}

The spherical functions on $\Omega$ may be defined for $\bsn$ in $\fa_\C^*$ by the formula 
$$\Phi_\bsn(x)=\int_{K}{\Dn(kx)dk}.$$ 
When $\bsn$ satisfies $\Re\nu_1\geq\dots\geq\Re\nu_r\geq0$, a property that we 
will denote by $\Re\bsn\geq0$, the generalised power function
$\Delta_{\bsn}$ and the spherical function $\Phi_\bsn$ extend continuously to $\ov\Omega$.
Now let 
\begin{equation}\label{D:al}
\fa_l=\langle L(c_j),j=1,\dots, l\rangle
\end{equation}
for $1\leq l\leq r-1$ and assume that $\bsn$ belongs to 
$({\fa_l}_\C)^*$,
i.e. that its $(r-l)$ last coordinates vanish. Then $\bsn$ also defines a spherical function 
$\Phi^{(l)}_\bsn$ of $\Omega^{(l)}$. 
Let $\alpha$ be a real number. When $\bsn$ appears in the argument of an object related to $V^{(l)}$,
we will  use the convention that $\bsn+\alpha:=(\nu_1+\alpha,\dots,\nu_l+\alpha)$. Recall that $\Ol\subset\pol$. 
\begin{theo}\label{T:En}
Let $\bsn$ in ${\fa_l}_\C^*$ such that $\Re\bsn\geq0$. Then for all $x$ in $\Ol$,
$$\Phi_{\bsn}(x)=\gamma_{\bsn}^{(l)}\Phi^{(l)}_{\bsn}(x), 
\quad\text{where}\quad  \gamma_{\bsn}^{(l)}=
\frac{\Gamma_{\Ol}(\frac{rd}{2})\Gamma_{\Ol}(\bsn+\frac{ld}{2})}
{\Gamma_{\Ol}(\frac{ld}{2})\Gamma_{\Ol}(\bsn+\frac{rd}{2})}.$$ 
\end{theo}
\noindent Here $\Gamma_{\Ol}$ is the Gindikin Gamma function for the cone $\Ol$,
$$\Gamma_{\Ol}(\bsn)=(2\pi)^{\frac{l(l-1)d}{4}}\prod^{l}_{j=1}\Gamma\left(\nu_j-(j-1)\tfrac{d}{2}\right).$$
The theorem is proved in the case $\bsn\in\N^l$ in 
\cite[Proposition 1.3.2 and remark 1.3.4]{ARUP}. We use this result and the following lemma, 
which is based on Blaschke's theorem (see \cite[Lemma A.1]{BK} for a detailed proof).
\begin{lem}
Let $f$ be a holomorphic function defined on the right half-plane 
$\{z\in\C\mid \Re z>0\}$. If $f$ is bounded and $f(n)=0$ for $n\in\N$, then $f$ is identically zero. 
\end{lem}
 
\begin{proof}[Proof of the theorem]
Let us set $z_j=\nu_j-\nu_{j+1}$, $j=1,\dots, l-1$, $z_l=\nu_l$, so that 
$\Re z_j\geq 0$ and $\nu_j=\sum_{k=j}^lz_k$. Let $x\in\Ol$ and let 
$$F(z_1,\dots,z_l)=\Phi_{\bsn(z_1,\dots,z_l)}(x)-\gamma_{\bsn(z_1,\dots,z_l)}^{(l)}
\Phi^{(l)}_{\bsn(z_1,\dots,z_l)}(x).$$
Let us fix $z_j=m_j\in\N$, $j=2,\dots,l$.
If $b\geq a>0$, one can see by Stirling's formula that $\frac{\Gamma(z+a)}{\Gamma(z+b)}$ 
is bounded on the right half plane. It follows that
the function 
$$z\mt\frac{\Gamma(z+\sum_{k=2}^{l}m_k+\frac{ld}{2})}{\Gamma(z+\sum_{k=2}^{l}m_k+\frac{rd}{2})}$$ 
is bounded and hence also $z\mt\gamma_{\bsn(z,m_2,\dots,m_l)}^{(l)}$.
Now
\begin{align*}
\nm{\Phi_{\bsn(z,m_2,\dots,m_l)}(x)}&\leq\int_{K}{\nm{\Delta_{(1)}^z(kx)\Delta_{(2)}^{m_2}(kx)\dots
\Delta_{(l)}^{m_l}(kx)}dk}\\
&\leq\sup_K\left(\Delta_{(2)}^{m_2}(kx)\dots\Delta_{(l)}^{m_l}(kx)\right)
\left(\sup_K\Delta_{(1)}(kx)\right)^{\Re(z)},
\end{align*}
and
$$\nm{\Phi^{(l)}_{\bsn(z,m_2,\dots,m_l)}(x)}\leq\sup_{K^{(l)}}\left(\Delta_{(2)}^{m_2}(kx)\dots
\Delta_{(l)}^{m_l}(kx)\right)\left(\sup_{K^{(l)}}\Delta_{(1)}(kx
)
\right)^{\Re(z)}$$
Let $\delta>\sup_K\Delta_{(1)}(kx)\geq\sup_{K^{(l)}}\Delta_{(1)}(kx)$. 
Then the holomorphic function $f(z)=F(z,m_2,\dots,m_l)\delta^{-z}$ is bounded and vanishes on $\N$, hence on 
the right half plane, i.e., for every $z\in\C$ with $\Re z>0$ and $m_j\in \N$, 
$$F(z,m_2,\dots,m_l)=0.$$
By the same argument one shows that for every $z_1\in\C$ with $\Re z_1>0$ and $m_j\in \N$, 
the map $z\mt F(z_1,z,m_3,\dots,m_l)$ vanishes identically, and
the proof follows by induction.
\end{proof}

\section{A series of spherical unitary representations}\label{S:sr}

In this section we introduce a family of spherical unitary representations that 
will occur in the decomposition of $L^2(\pol)$ under the 
action of $G$.

For $1\leq l\leq r-1$ let
$$\ov{\fn}_l=\bigoplus_{l\geq i<j}\fg^{-}_{ij}.$$
Note that it is a (nilpotent) Lie algebra and that (cf. \eqref{D:al})
$$\fa_l\oplus\R L(e)=\bigcap_{l\leq i<j}\ker\alpha^\pm_{ij}.$$
The closed subgroup $Z_G(\fa_l)=Z_G(\fa_l\oplus\R L(e))$ normalises $\ov{N}_l=\exp\ov{\fn}_l$ hence  
$$Q_l=Z_G(\fa_l)\ov{N}_l$$
is a subgroup of $G$. Moreover $Z_G(\fa_l)\cap\ov{N}_l=\{\id\}$ so this decomposition 
is a semidirect product.
The Lie algebra of $Q_l$,
$$\fq_l=\fz_\fg(\fa_l)\oplus\ov{\fn}_l=\fm\oplus\fa\oplus\bigoplus_{l<i<j}\fg^\pm_{ij}\oplus\bigoplus_{l\geq i<j}\fg^{-}_{ij},$$
where $\fm=\fz_\fk(\fa)$,
is a parabolic subalgebra of $\fg$, and since $Q_l$ is the normaliser of $\fq_l$ in $G$ \cite[7.83]{KNA}, 
it is a closed subgroup of $G$ 
(the parabolic subgroup associated to $\fq_l$). It is also the stabiliser of a flag 
of idempotents $(e_1,e_2,\dots,e_l)$. 
Note also that since $V(L(c_j),1)=\R c_j$, we have  
\begin{equation}\label{E:actcent}
Z_G(\fa_l)c_j=\R_+ c_j,\quad j=1,\dots, l. 
\end{equation}

\begin{lem}Let $A_l=\exp\fa_l$ and 
$$M_l=\cap_{j=1}^lZ_G(\fa_l)_{c_j}.$$
Then the multiplication map
$$M_l\times A_l\times \ov{N}_l\ra Q_l$$ 
is a diffeomorphism.
\end{lem}
\begin{proof}
It is clear from \eqref{E:actcent} that the product map $M_l\times A_l\ra Z_G(\fa_l)$ is a smooth bijection 
hence it is a diffeomorphism. 
\end{proof}
\noindent Note that the decomposition in the preceding lemma is not exactly the Langlands decomposition of $Q_l$. 
However, it is more adapted to our purpose. We will let $a_l(q)$ denote the $A_l$-component of 
$q\in Q_l$ in the preceding decomposition.
For $\bsn\in({\fa_l}_\C)^*$, let $1\otimes e^\bsn\otimes1$ be the character of $Q_l$ defined by $(1\otimes 
e^\bsn\otimes1)(q)=e^{\bsn\log a_l(q)}$, and
let us denote by $C(G,Q_l,1\otimes e^{\bsn}\otimes 1)$ 
the Frechet space of continuous complex valued functions on $G$ that are 
$Q_l$-equivariant with respect to $1\otimes e^{\bsn}\otimes 1$, i.e., 
$$C(G,Q_l,1\otimes e^{\bsn}\otimes 1)=\{f\in C(G)\mid \forall q\in Q_l,\ f(gq)=
e^{-\bsn\log{a_l(q)}}f(g)\}.$$
The induced representation $\Ind_{Q_l}^{G}(1\otimes e^{\boldsymbol\nu}\otimes1)$ is the left 
regular representation of $G$ on 
$C(G,Q_l,1\otimes e^{\bsn}\otimes 1)$.

We will know determine values of $\bsn\in({\fa_l}_\C)^*$ for 
which the representation 
$$\Ind_{Q_l}^{G}(1\otimes e^{\boldsymbol\nu}\otimes 
1)\otimes\Delta^{-\frac{ld}{4}}$$
(cf. \eqref{D:char}) can be made unitary and irreducible.

The group $G$ 
admits the Levi decomposition
$$G=G'\times\R_{+}$$
where the semisimple part $G'$ is the kernel of the 
character $\Delta$. Then $K$ is a maximal compact subgroup of $G'$ and the Lie algebra 
$\fg'$ of $G'$ has Cartan decomposition $\fg'=\fk\oplus\fp'$ 
with $\fp'=\{L(x)\in\fp\mid \tr x=0\}$, and 
$\fa'=\fa\cap\fp'$ is maximal abelian in $\fp'$. Let
$$\fa_l'=\bigoplus_{1\leq j\leq l}\R \left(L(c_j)-\frac{L(e-e_l)}{r-l}\right)=\bigcap_{l\leq i<j}\ker\alpha^\pm_{ij},$$
$$\fm_l'=\fm\oplus\bigoplus_{j=l+1}^{r-1}
\R 
\left(L(c_j)-\frac{L(e-e_l)}{r-l}\right)\oplus
\bigoplus_{l<i<j}\fg^\pm_{ij}.$$
Then
$$\fq'_l=\fm_l'\oplus \fa_l'\oplus \ov{\fn}_l.$$
is a parabolic subalgebra of $G'$. The corresponding parabolic subgroup $Q'_l$ admits the 
Langlands decomposition 
$$Q_l'=M_l'A_l'\ov{N}_l,$$
where $A_l'=\exp\fa_l'$ and (cf. \cite[Ch. VII, Propositions 7.25, 7.27 and 7.82]{KNA})
$$M_l'=Z_K(\fa_l')\exp(\fm_l'\cap\fp'),$$
whose Lie algebra is $\fm'_l$.
For $\bsn\in({\fa'_l}_\C)^*$ the induced representation 
$$\Ind_{Q'_l}^{G'}(1\otimes e^{\boldsymbol\nu}\otimes 1)$$ 
is defined in the same way as for $G$.

\begin{lem}
\begin{enumerate}[(i)]
\item $Q_l= Q_l'\times\R_{+},$
\item $M_l'\subset M_l.$
\end{enumerate}
\end{lem}
\begin{proof}
To prove (i) we observe that since $\R_+\subset Q_l$, we can write $Q_l=Q'\times\R_{+}$, 
for some subgroup $Q'\subset G'$. Since $Q_l$ 
(resp. $Q_l'$) 
is the normaliser of $\fq_l$ (resp. $\fq_l'$) in $G$ (resp. 
$G'$), the inclusion $Q_l'\subset Q'$ is obvious and the converse follows 
from the fact that $\Ad(G)$ preserves $\fg'$.
Since $X.c_j=0$ for $X\in\fm_l'$ and $1\leq j\leq l$, the assertion (ii) will follow from 
$Z_K(\fa_l)=Z_K(\fa_l')$. Let 
$k\in Z_K(\fa_l')$, i.e., for $j=1,\dots, l$,
\begin{equation}\label{E:kframe}
k.\left(L(c_j)-\frac{L(e-e_l)}{r-l}\right)=L(kc_j)-\frac{L(e-ke_l)}
{r-l}=L(c_j)-\frac{L(e-e_l)}{r-l}. 
\end{equation}
By summing over $j$ one obtains
$$L(ke_l)-\frac{l}{r-l}{L(e-ke_l)}=L(e_l)-
\frac{l}{r-l}{L(e-e_l)},$$
and hence $L(ke_l)=L(e_l)$. By \eqref{E:kframe} we then have $L(kc_j)=L(c_j)$ for $j=1,\dots, l$, i.e., $k\in Z_K(\fa_l)$.
\end{proof}

Let us denote by 
$(\ti\delta_j)_{j=1\dots l}$ the dual basis of 
$(L(c_j)-\frac{L(e-e_l)}{r-l})_{j=1\dots l}$ in $({\fa'_{l}}_\C)^*$. By $\delta_j\mt\ti{\delta}_j$ 
we define an isomorphism 
$({\fa_{l}}_\C)^*\simeq({\fa'_{l}}_\C)^*$, $\bsn\mt\ti\bsn$.

Let $m_\bsn$ be the character of $\R_{+}$ defined by
\begin{equation}\label{D:mchar}
m_\bsn(\zeta)=\zeta^{-\frac{rld}{4}+\sum_{1\leq j}{\nu_j}}.
\end{equation}

\begin{pro}\label{P:isorep}
$$\Ind_{Q_l}^{G}(1\otimes e^\bsn\otimes 1)\otimes
\Delta^{-\frac{dl}{4}}\simeq\Ind_{Q_l'}^{G'}(1\otimes 
e^{\ti\bsn}\otimes 1)\otimes m_{\bsn}$$
\end{pro} 

\begin{proof}
Let us denote by $a_l(q')$ the $A'_l$-component of $q'$  
in the Langlands decomposition of $Q_l'$. Then for all $q=q'\xi\in Q_l=Q'_l\times\R_{+}$,
\begin{equation}\label{E:Qeq}
e^{\bsn\log a_l(q)}=\xi^{\sum{\nu_j}}e^{\ti\bsn\log a'_l(q')}.
\end{equation}
Indeed, if $q'=m'e^{\sum{s_j(L(c_j)-\frac{L(e-e_l)}{r-l}})}\ov{n}$
is the Langlands decomposition of $q'$ in $Q_l'$, then 
$$q=(m'e^{(-\frac{l}{r-l}+\log{\xi})L(e-e_l)})e^{\sum{s_jL(c_j)}+\log{\xi}L(e_l)}\ov{n}$$
is the Langlands decomposition of q in $Q_l$. 
If $f\in C(G,Q_l,1\otimes e^{\bsn}\otimes1)$, then by \eqref{E:Qeq}, its 
restriction $\ti f$ 
to $G'$ belongs to the space $C(G',Q_l',1\otimes e^{\ti\bsn}\otimes1)$. Conversely, 
if $\ti{f}\in C(G',Q_l',1\otimes e^{\ti\bsn}\otimes1)$, one obtains, again by \eqref{E:Qeq}, 
a function $f\in C(G,Q_l,1\otimes e^{\bsn}\otimes1)$ by setting 
$f(g'\zeta):=\zeta^{-\sum_j \nu_j}\ti{f}(g')$ and 
we obtain thereby a bijection 
$C(G,Q_l,1\otimes e^{\bsn}\otimes1)\simeq C(G',Q_l',1\otimes e^{\ti\bsn}\otimes1)$. 
Now the 
operator
\begin{align*}
\mathcal T: C(G,Q_l,1\otimes e^{\bsn}\otimes1)\otimes\C & \ra
 C(G',Q_l',1\otimes e^{\ti\bsn}\otimes1)\otimes\C,\\
f \otimes z & \mt \ti{f}\otimes z,
\end{align*}
intertwines the actions of $G$. Indeed, if $f\in 
C(G,Q_l,1\otimes e^{\bsn}\otimes1)$ and $g=g'\zeta\in G=G'\times\R_{+}$ then 
$\ti{f(g^{-1}\cdot)}=\ti{f(g'^{-1}\zeta^{-1}\cdot)}=\zeta^{
\sum{\nu_j}}\ti{f}(g'^{-1}\cdot)$, and hence
\begin{align*}
\mathcal{T}\left(f(g^{-1}\cdot)\otimes\Delta^{-\frac{dl}{4}}(g)z\right)
&=\ti{f(g^{-1}\cdot)}\otimes \Delta^{-\frac{ld}{4}}(g)z\\
&=\zeta^{\sum{\nu_j}}\ti{f}(g'^{-1}\cdot)\otimes \zeta^{-
\frac{rld}{4}}z\\
&=\ti{f}(g'^{-1}\cdot)\otimes \zeta^{-
\frac{rld}{4}+\sum{\nu_j}}z.
\end{align*} 
\end{proof}
 
The representation $\Ind_{Q_l}^{G}(1\otimes e^\bsn\otimes 1)\otimes
\Delta^{-\frac{ld}{4}}$ extends to a continuous representation (denoted by the same symbol) on the 
Hilbert completion of $C(G,Q_l,1\otimes e^\bsn\otimes1)\otimes\C$ 
with respect to
$$\nm{f\otimes z}^2=\int_{K}{f(k)dk}\nm{z}^2.$$

\begin{lem}\label{L:invMl} For any $m\in M_l$ and any $x\in V$,
$$\Delta_{(j)}(mx)=\Delta_{(j)}(x),\quad 1\leq j\leq l.$$
\end{lem}
\begin{proof}
Let $m\in M_l$. Then for $j=1,\dots, l$, $m$ commutes with $L(e_j)$, and $m\in\Aut{V^{(j)}}$, so 
\begin{align*}
\Delta_{(j)}(mx)=\Delta^{(j)}(P(c_j)(mx))=\Delta^{(j)}(mP
(c_j)(x))&=\Delta^{(j)}(P(c_j)(x))\\
&=\Delta_{(j)}(x).
\end{align*}
\end{proof}

\begin{pro} 
The map $g\mt\Delta_{-\bsn}(g^*e)$ is a (norm one) $K$-invariant vector in 
$C(G,Q_l,1\otimes e^\bsn \otimes1)$.
\end{pro}
\begin{proof}
Let $q\in Q_l$ and let 
$q=man\in M_lA_l\ov{N}_l$. 
Then for $g\in G$, one 
has, since $n^*\in N$ and $a^*=a$,
$$\Delta_{-{\boldsymbol\nu}}((gq)^*e)=\Delta_{-{\boldsymbol\nu
}}(n^*a^*m^*g^*e)=e^{-{\boldsymbol\nu}\log(a)}
\Delta_{-{\boldsymbol\nu}}(m^*g^*e)$$
and because of Lemma \ref{L:invMl},
$$\Delta_{-{\boldsymbol\nu}}((gq)^*e)=e^{-{\boldsymbol\nu}\log
(a)}\Delta_{-{\boldsymbol\nu}}(g^*e).$$
\end{proof}

\noindent Let 
$$\bsr_l=\frac{d}{2}\sum_{l \geq i <j}{\frac{\delta_i-\delta_j}{2}}=
\frac{d}{4}\sum_{j=1}^{l}{(r+1-2j)}\delta_j$$
be the half sum of the negative $\fa_l\oplus\R L(e)$-restricted roots 
(counted with multiplicities).
\begin{theo}
For almost every $\bsl\in\fa_l^*$ (with respect to Lebesgue measure), the representation 
$\Ind_{Q_l}^{G}(1\otimes e^\bsn\otimes 1)\otimes
\Delta^{-\frac{ld}{4}}$ with $\bsn=i\bsl+\bsr_l+\frac{ld}{4}$ is an irreducible unitary 
spherical representation. 
\end{theo}
\begin{proof}
First, let us remark that if $\bsn=i\bsl+\bsr_l+\frac{ld}{4}$ with $\bsl\in\fa_l^*$, 
then since
$$\sum{\nu_j}-\frac{rld}{4}=i\sum{\lambda_j}+
\frac{d}{4}\left(l(r+1+l)-2\frac{l(l+1)}{2}-rl\right)=i\sum{\lambda_j},$$
the representation $m_\bsn$ (cf.~\eqref{D:mchar}) is unitary.
We now claim that $\ti{\bsr_l+\frac{dl}{4}}$ is the 
half sum of the negative $\mathfrak{a}'_l$-restricted roots (counted with 
multiplicities).
Indeed,
\begin{align*}
\frac{d}2\sum_{l\geq i<j}\frac{\delta_i-\delta_j}{2}&
\left(L(e_k)-\frac{L(e-e_l)}{r-l}\right)\\
&=\frac{d}{4}(r+1-2k)+\frac{d}2
\sum_{i\leq l<j}
\frac{\delta_i-\delta_j}{2}\left(-\frac{L(e-e_l)}{r-l}\right)\\
&=\frac{d}{4}(r+1-2k)+\frac{d}4\sum_{i\leq l<j}{(-\delta_j)}\left(-
\frac{1}{r-l}\sum_{k>l}{L(c_k)}\right)
\\
&=\frac{d}{4}(r+1-2k)+\frac{ld}{4}=\ti{\bsr_l+\frac{ld}{4}}\left(L(e_k)-\frac{L(e-e_l)}{r-l}\right).
\end{align*}
The theorem now follows from Proposition \ref{P:isorep} and Bruhat's theorem \cite[Theorem 2.6]{VdB}.
\end{proof}
\noindent Let us note $\bsr_l'=\bsr_l+\frac{ld}{4}$. We now set 
$$(\pi_\bsn,\ch_\bsn):=\Ind_{Q_l}^{G}(1\otimes 
e^{i\bsn+\bsr'_l}\otimes 1)\otimes
\Delta^{-\frac{ld}{4}},$$
and
$$v_{\bsn}:=\Delta_{-(i\bsn+\bsr'_l)}((\cdot)^*e)\otimes1.$$ 

Let us compute the positive definite spherical function associated to $\pi_\bsn$, that is,
$$\Phi(g)=\langle\pi_\bsn(g)v_\bsn,v_{\bsn}\rangle_\bsn=
\int_{K}{\Delta_{-(i\bsn+\bsr'_l)}((g^{-1}k)^*e)dk}\Delta^{-\frac{ld}{4}}(g).$$ 
Since $g^{-*}e=\Theta(g)e=(ge)^{-1}$ \cite[Theorem III.5.3]{FAKO}, $\Phi$ can be written, 
as a function on $\Omega$,
$$\Phi(x)=\int_{K}{\Delta_{-(i\bsn+\bsr'_l)}(x^{-1})dk}\Delta^{-\frac{ld}{4}}(x)=
\Phi_{-(i\bsn+\bsr'_l)}(x^{-1})\Delta^{-\frac{ld}{4}}(x),$$
and since $\Phi_{-(i\bsn+\bsr'_l)}(x^{-1})=\Phi_{i\bsn+\bsr'_l+2\bsr}(x)$
with $\bsr=\sum_{j=1}^r{(2j-r-1)\delta_j}$ \cite[Theorem XIV.3.1 (iv)]{FAKO},
$$\Phi(x)=\Phi_{i\bsn+\boldsymbol\eta_l+\bsr}(x)\quad\text{where}\quad \boldsymbol\eta_l=\frac{d}{4}\sum_{j=l+1}^r{(2j-l-r-1)\delta_j}.$$
Recall \cite[Theorem XIV.3.1 (iii)]{FAKO} that $\Phi_{\bsn'+\bsr}=\Phi_{\bsn+\bsr}$ if and only if $\bsn'=w\bsn$ for $w\in W$. 
Hence the representations 
$\pi_\bsl$ and $\pi_{\bsl'}$, with $\bsl$, $\bsl'$ in $\fa_l^*$, are equivalent if and 
only if $i\bsl'+\boldsymbol\eta_l=w(i\bsl+\boldsymbol\eta_l)$. Since $\boldsymbol\eta_l$ is 
real it follows that $\pi_{\bsl'}$ is equivalent to $\pi_\bsl$ if and only if $\bsl'=w\bsl$ with 
$w\in W_l:=\mathfrak{S}_l$.

\section{The intertwining operator and the Plancherel formula}

Let $\bsn\in\C^l$ such that $\Re\left(-(i\bsn+\bsr'_l)\right)\geq0$. Then for 
$f\in\cpol$ and $g\in G$, the formula
\begin{equation}\label{D:intop}
T_\bsn f(g)=\int_{\pol}{f(x)\Delta_{-(i\bsn+\bsr'_l)}(g^*x)d\mu_l(x)}
\end{equation}
defines a continuous function on $G$. Moreover, it follows from \eqref{E:hcef} and Lemma~\ref{L:invMl} that
\begin{equation*}
T_\bsn f\in C(G,Q_l,1\otimes e^{i\bsn+\bsr'_l}\otimes1). 
\end{equation*}


\noindent We will also view $T_{\boldsymbol\nu}$ as an operator with 
values in $C(G,Q_l,1\otimes e^{i\bsn+\bsr'_l}\otimes1)\otimes\C$ (in the obvious way), and hence 
in $\ch_\bsn$.

\begin{lem}\label{L:inter}
For  $\Re\left(-(i\bsn+\bsr'_l)\right)\geq0$, the operator
$$T_{\boldsymbol\nu}:C^\infty_0(\pol)\ra \ch_\bsn$$
intertwines $\pi^{l}$ and $\pi_\bsn$. 
\end{lem}
\begin{proof}
Let us set $\bsn'=-(i\bsn+\bsr'_l)$.
Let $h\in G$. Then
\begin{align*}
T_{\boldsymbol\nu}(h.f)(g)&=\int_{\partial\Omega_l}
{\Delta^{\frac{ld}{4}}(h) f(h^*x)
\Delta_{\bsn'}(g^*x)\dmul(x)}\\
&=\Delta^{-\frac{ld}{4}}(h)
\int_{\partial\Omega_l}{f(x)\Delta_{\bsn'}(g^*h^{-*}x)\dmul(x)}\\ 
&=\Delta^{-\frac{ld}{4}}(h)\int_{\partial\Omega_l}
{f(x)\Delta_{\bsn'}((h^{-1}g)^*x)\dmul(x)},
\end{align*}
i.e.,
\begin{equation}
T_{\boldsymbol\nu}(h.f)(g)=\Delta^{-\frac{ld}{4}}(h)
T_{\boldsymbol\nu}(f)(h^{-1}g). \label{E:inter}
\end{equation}
\end{proof}

Since $\bsr'_l=\frac{d}{2}\sum_{j=1}^l(r+l+1-2j)\delta_j$, 
we do not have $\Re((-i\bsl+\bsr'_l))\geq0$ when $\bsl\in\fa_l^*$, and the integral \eqref{D:intop} does not 
converge. This means that the integral has to be interpreted in a
suitable sense using analytic continuation in the parameter $\bsn$. For this we recall 
that when $\bsn\in\C^l$, the Riesz distribution $R_{\bsn+\frac{ld}{2}}$ 
on $V$ can be defined as the analytic continuation of the following integral 
$$R_{\bsn+\frac{ld}{2}}(F)=\Gamma_{\Ol}\left(\bsn+\tfrac{ld}{2}\right)^{-1}
\int_{\pol}{F(x)\Dn(x)\dmul(x)},\quad F\in\mathcal{S}(V),$$
where $\mathcal{S}(V)$ is the Schwartz space of $V$, and that it has support in $\ov{\pol}$ 
(cf. \cite[Theorem 5.1 and 5.2]{ISHI}, where the 
integral is actually defined over $\mathcal{O}_l=\{x \in \pol \mid \Delta_{(l)}(x) \neq 0\}$,
but $\mu_l(\pol\setminus\mathcal{O}_l)=0$). The restriction (denoted by the same symbol) to the 
open set $V_{\geq l}$ is then 
a distribution with support in the submanifold $\pol$. We can therefore consider the vertical restrictions
$R_{\bsn+\frac{ld}{2}}\mid_{\pol}$ (cf. Appendix A). Since $R_{\bsn+\frac{ld}{2}}$ is a measure with 
support on $\pol$ for $\Re\bsn \geq 0$, these restrictions do not depend on the choice of 
a tubular neighbourhood (cf. Proposition~\ref{P:invres}).
For $\Re(-(i\bsn+\bsr_l'))\geq 0$, we have
\begin{equation*}
T_{\bsn}f(g)=\Gamma_{\Ol}\left(-(i\bsn+\bsr'_l)+\tfrac{ld}2\right)
\Delta^{-\frac{ld}{2}}(g)\left(R_{-(i\bsn+\bsr'_l)+\frac{ld}{2}}\right)\mid_{\pol}\left(f(g^{-*}\cdot)\right). 
\end{equation*}
Hence, we can let the right hand side define an analytic continuation of the integrals 
$T_{\bsn}f(g)$. It is defined on the complement $\mathcal{Z}$
of the set of poles of the meromorphic function $\Gamma_{\Ol}\big(-(i\bsn+\bsr'_l)\big)$.
Since for fixed $f \in \cpol$ the map $g\mt f(g^{-*}\cdot)$ is continuous, 
the function $g \mapsto T_{\bsn}(f)(g)$ is continuous. 

\begin{pro}
For any $\bsn \in\mathcal{Z}$, 
$$T_\bsn f\in C(G,Q_l,1\otimes e^{i\bsn+\bsr'_l}\otimes1),$$
and the operator 
$$T_{\boldsymbol\nu}:C^\infty_0(\pol)\ra \ch_\bsn$$
intertwines $\pi^{l}$ and $\pi_\bsn$.
\end{pro}

\begin{proof}
The equation describing the $Q_l$-equivariance as well as eq.~\eqref{E:inter} 
are analytic in the parameter
$\bsn$. Hence they hold by analytic continuation since they hold on the open set 
where $\Re(-(i\bsn+\bsr'_l))>0$.
\end{proof}


We now recall, in order to fix the notations, the 
definition of the spherical Fourier transform on $\Ol$. 
If $f$ is a continuous function with compact support on 
$\Omega^{(l)}$ which is $K^{(l)}$-invariant, its spherical 
Fourier transform is
$$\widehat{f}(\bsn)=\int_{\ol}{f(x)
\Pl_{-\bsn+\bsr^{(l)}}(x)\dxl},$$
where $\bsn\in ({\fa_l}_{\C})^*$ and 
$\bsr^{(l)}=\frac{d}{4}\sum_{j=1}^{l}{(2j-l-1)\delta_j}$. 
Since $f$ has compact support, the function $\widehat f$ is 
holomorphic on $({\fa_l}_{\C})^*$. For latter use we also 
recall the inversion formula for $f\in 
C_0^{\infty}(\Omega^{(l)})^{K^{(l)}}$ and $\bsl\in\fa_l^*$ 
(cf. \cite[Theorem XIV.5.3]{FAKO} and \cite[Ch. III, Theorem 7.4]{HEL}):
\begin{equation}\label{E:if}  
f(x)=c^{(l)}_0\int_{\fa_l^*}{\widehat{f}(i\bsl)
\Phi^{(l)}_{i\bsl+\bsr^{(l)}}(x)
\frac{d\bsl}{\nm{c^{(l)}(\bsl)}^2}},
\end{equation}
where $d\bsl$ is the Lebesgue measure on $\fa_l^*\simeq\R^l$, 
$c^{(l)}(\bsl)$ is Harish Chandra's $c$-function for 
$\Omega^{(l)}$, and $c^{(l)}_0$ is a positive constant.

Now let $f\in\cpol^K$\!, and observe that 
$\Omega^{(l)}$, being a fibre of $\pol\ra\Pi_l$, is closed in 
$\pol$, and hence $f\mid_{\Omega^{(l)}}$ (sometimes still denoted by $f$) 
has compact support in 
$\Omega^{(l)}$. Moreover, since any $k\in K^{(l)}$ 
extends to an element of $K$, the function 
$f\mid_{\Omega^{(l)}}$ is $K^{(l)}$-invariant.
\begin{pro}
If $f\in C^\infty_0(\pol)$ is 
K-invariant then
\begin{equation}
T_\bsn f(g)=\gamma_{-(i\bsn+\bsr'_l)}^{(l)}\widehat{f}(i\bsn-\tfrac
{rd}{4})\Delta_{-(i\bsn+\bsr'_l)}(g^*e) 
\end{equation}
\end{pro}
\begin{proof}
Let us set $\bsn'=-(i\bsn+\bsr'_l)$ in the following.
Again, by analytic continuation, it suffices to prove the equality 
for $\Re(\bsn')\geq 0$.
Since $f$ and $\mu_l$ 
are $K$-invariant one has for all $k$ in $K$,
$$T_\bsn f(g)=\int_{\partial\Omega_l}{f(k^{-1}x)
\Delta_{{\boldsymbol\nu}'}(g^*x)\dmul(x)}
=\int_{\partial\Omega_l}{f(x)
\Delta_{{\boldsymbol\nu}'}(g^*kx)\dmul(x)}.$$
Hence
$$
T_\bsn f(g)=\int_K{T_\bsn f(g)dk}=\int_{\partial\Omega_l}{f(x)
\left(\int_K{\Delta_{{\boldsymbol\nu}'}(g^*kx)dk}
\right)\dmul(x)}.
$$
Writing $g^*=th$, $t\in NA$, $h\in K$, we have 
$$\Delta_{{\boldsymbol\nu}'}(thkx)=\Delta_{{
\boldsymbol\nu}'}(hkx)\Delta_{{
\boldsymbol\nu}'}(te)=\Delta_{{
\boldsymbol\nu}'}(hkx)\Delta_{{
\boldsymbol\nu}'}(g^*e),$$
and using the left invariance of the Haar measure of $K$,
\begin{align*}
\int_K{\Delta_{{\boldsymbol\nu}'}(g^*kx)dk}&=\int_K{\Delta_{{
\boldsymbol\nu}'}(hkx)dk}\Delta_{{
\boldsymbol\nu}'}(g^*e)\\
&=\int_K{\Delta_{{
\boldsymbol\nu}'}(kx)dk}\Delta_{{
\boldsymbol\nu}'}(g^*e)\\
&=\Phi_{\bsn'}(x)
\Delta_{{\boldsymbol\nu}'}(g^*e),
\end{align*}
hence,
$$T_\bsn f(g)=\int_{\Omega^{(l)}}{f(x)\Phi_{\bsn'}(x)}\dmul(x) 
\Delta_{{\boldsymbol\nu}'}(g^*e).$$
Now Upmeier and Arazy's polar decomposition \eqref{E:pd} for $\mu_l$ 
yields
$$T_\bsn f(g)=\int_{\Omega^{(l)}}{f(x)\Phi_{\bsn'}(x)
\Delta_{(l)}^{\frac{rd}{2}}(x)}\dxl 
\Delta_{{\boldsymbol\nu}'}(g^*e),$$
and by Theorem \ref{T:En},
\begin{align*}
T_\bsn f(g)=&\gnlp\int_{\Omega^{(l)}}{f(x)\Phi^{(l)}_{\bsn'}(x)
\Delta_{(l)}^{\tfrac{rd}{2}}(x)}
\dxl \Delta_{{\boldsymbol\nu}'}(g^*e)\\
&=\gnlp\int_{\Omega^{(l)}}{f(x)\Phi^{(l)}_{\bsn'+
\tfrac{rd}{2}}(x)}\dxl \Delta_{{\boldsymbol\nu}'}(g^*e)\\
&=\gamma_{\bsn'}^{(l)}\widehat{f}(-\bsn'+\bsrl-
\tfrac{rd}{2})\Delta_{\bsn'}(g^*e).
\end{align*} 
Since $-\bsr_l'+\bsr^{(l)}-\frac{rd}{2}=-\frac{rd}{4}$, we eventually get
$$T_\bsn f(g)=\gamma_{-(i\bsn+\bsr'_l)}^{(l)}\widehat{f}(i\bsn-
\tfrac{rd}{4})\Delta_{-(i\bsn+\bsr'_l)}(g^*e).$$
\end{proof}

\noindent In the following we set,
$$\ti f(\bsn)=\gamma_{-(i\bsn+\bsr'_l)}^{(l)}\widehat{f}(i\bsn-\tfrac
{rd}{4}),$$
and we note that it defines a meromorphic function whose 
poles are those 
of $\Gamma_\Omega(-(i\bsn+\bsr'_l)+\frac{ld}{2})$.
The following lemma will be used in the proof of the 
Plancherel formula.

\begin{lem}[Inversion formula]\label{L:if2} Let $f\in 
\cpol^K$and let $\bsl\in \fa_l^*$. Then 
for $x\in\Omega^{(l)}$,
\begin{equation*}
f(x)=c_0^{(l)}\int_{\fa_l^*}{\ti{f}(\bsl)
\Phi^{(l)}_{i\bsl+\bsr^{(l)}-\frac{rd}{4}}(x)
(\gamma^{(l)}_{-(i\bsl+\bsr_l' )})^{-1}
\frac{d\bsl}{\nm{c^{(l)}(\bsl)}^2}}.
\end{equation*}
\end{lem}
\begin{proof}
We have
\begin{equation}\label{E:pmir}
\widehat{f\Delta_{(l)}^{\frac{rd}{4}}}(i\bsl)=
\widehat{f}(i\bsl-\tfrac{rd}{4})=\ti{f}(\bsl)
(\gamma^{(l)}_{-(i\bsl+\bsr'_l )})^{-1} ,
\end{equation}
hence the inversion formula $\eqref{E:if}$ applied to the 
function $f\Delta_{(l)}^{\frac{rd}{4}}$ gives the desired 
formula. 
\end{proof}

We now state the main result of the article. Recall the notations from the end of section~\ref{S:sr}.
\begin{theo}[The Plancherel Theorem]\label{T:Pl}
Let $p$ be the measure on $\fa_l^*/W_l$ defined by 
$$dp(\bsl)=\frac{c_0^{(l)}}{\nm{\gamma^{(l)}_{-(i
\bsl+\bsr_l')
}c^{(l)}(\bsl)}^2}d\bsl.$$
Then there exists
an isomorphism of unitary representations
$$T:\left(\pi^{l},L^2(\pol)\right)\simeq\left(\int_{\fa_l^*}{\pi_
{
\bsl} dp(\bsl)},\int_{\fa_l^*}{\ch_{\bsl} dp(\bsl)}\right),$$
such that for every $f\in\cpol$, $(Tf)_\bsl=T_\bsl f$.
\end{theo}
\begin{proof}
First we prove that for any $K$-invariant function $f$ in 
$\cpol$,
\begin{equation}\label{E:pfk}
\int_{\pol}{\nm{f(x)}^2\dmul(x)}=\int_{\fa_l^*}
{\nm{\ti{f}(\bsl)}^2dp(\bsl)}.
\end{equation}
For this purpose we use the polar decomposition for $\mu_l$ and the 
inversion formula of Lemma \ref{L:if2}. Then
\begin{align*}
&\int_{\pol}{\nm{f(x)}^2\dmul(x)}=\int_{\Omega^{(l)}}{\nm
{f(x)}^2\Dl^{\frac{rd}{2}}(x)\dxl}\\ 
&=\int_{\Omega^{(l)}}{f(x)\Dl^{\frac{rd}{2}}(x)c_0^{(l)}\int_
{\fa_l^*}{\ov{\ti{f}(\bsl)}
{\Phi^{(l)}_{i\bsl+\bsr^{(l)}-\frac{rd}{4}}(x)}
{\ov{(\gamma^{(l)}_{-(i\bsl+\bsr_l' )})}}{}^{-1}
\frac{d\bsl}{\nm{c^{(l)}(\bsl)}^2}}\dxl}\\
&=\int_{\fa_l^*}{\ov{\ti{f}(\bsl)
(\gamma^{(l)}_{-(i\bsl+\bsr_l' )})}{}^{-1}
\left(\int_{\Omega^{(l)}}{f(x)\Dl^{\frac{rd}{4}}(x)\Phi^
{(l)}_{-i\bsl+\bsr^{(l)}}(x)\dxl}\right)
\frac{c_0^{(l)}d\bsl}{\nm{c^{(l)}(\bsl)}^2}}\\
&=\int_{\fa_l^*}{\ov{\ti{f}(\bsl)
(\gamma^{(l)}_{-(i\bsl+\bsr_l' )})}{}^{-1}
\widehat{f\Delta_{(l)}^{\frac{rd}{4}}}(i\bsl)
\frac{c_0^{(l)}d\bsl}{\nm{c^{(l)}(\bsl)}^2}}\\
&=\int_{\fa_l^*}{\nm{\ti{f}
(\bsl)}^2\frac{c_0^{(l)}d\bsl}{\nm{\gamma^{(l)}
_{-(i\bsl+\bsr_l')}c(\bsl)}^2}}.
\end{align*}
In the last equality we have used again the formula 
\eqref{E:pmir}.

The next step is to prove that for a dense subset of 
functions $f$ in $\cpol$, the identity  
\begin{equation*}
\int_{\pol}{\nm{f(x)}^2\dmul(x)}=\int_{\fa_l^*}
{\nm{T_\bsl f}_\bsl^2dp(\bsl)}
\end{equation*}
holds.

Recall that $L^1(G)$ is a Banach $*$-algebra when equipped 
with convolution
as multiplication, and $\ph^*(g):=\overline{\ph(g^{-1})}$.
Let $L^1(G)^{\#}$ denote the (commutative) closed subalgebra 
of left and right $K$
-invariant
functions in $L^1(G)^{\#}$. There is a natural projection 
$L^1(G) \rightarrow
L^1(G)^{\#}$,
\begin{equation*}
\ph \mapsto \ph^{\#}:=\int_K\int_K \ph(k_1^{-1} \cdot k_2)dk_1dk_2.
\end{equation*}
For a unitary representation $(\tau, \mathscr{H})$ of
$G$, there is a $*$-representation (also denoted by $\tau$) of
$L^1(G)$ on $\mathscr{H}$ given by
\begin{equation*}
\tau(\ph)v:=\int_G \ph(g)\tau(g)vdg, \quad v\in \mathscr{H}.
\end{equation*}
The representations of $K$ and $L^1(G)$ are related by
\begin{equation}
\tau(k_1) \tau(\ph)\tau(k_2)=\tau(\ph(k_1^{-1} \cdot k_2^{-1})), \quad 
\ph \in L^1(G), \quad k_1, k_2 \in K.
\label{E:kalg}
\end{equation}
The subspace $\mathscr{H}^K$ of $K$-invariants is invariant
under $L^1(G)^{\#}$.
From \eqref{E:kalg}, it follows that for any $\ph \in L^1(G)$, and 
$u,v \in \mathscr{H}^K$,
\begin{equation}
\langle \tau(\ph)u,v\rangle=\langle \tau(\ph^{\#})u,v\rangle. \label{E:projkinv}
\end{equation}

Let $\xi$ be the $K$-invariant cyclic vector in $L^2(\pol)$. 
We claim that 
there exists a sequence $\{\xi_n\}_{n=1}^{\infty} \subseteq 
\cpol^K$, such that
$\xi_n \rightarrow \xi$ in $L^2(\pol)$.
To see this, we can first choose a sequence 
$\{\zeta_n\}_{n=1}^{\infty} \subseteq \cpol$ that converges to
$\xi$. Next, observe that the orthogonal projection 
$P:L^2(\pol) \rightarrow L^2(\pol)$ is given by
$f \mapsto \int_K f(k^{-1} \cdot)dk$. Then $P(f)$ is smooth 
if $f$ 
is smooth. Moreover, supp $f$ is contained
in the image of the map $K \times \mbox{supp}\,f \rightarrow 
\pol$, 
$(k,x) \rightarrow kx$. It follows that $P(\cpol) \subseteq 
\cpol^K$.
Hence, the claim holds with $\xi_n:=P(\zeta_n)$.
The subspace $$\mathscr{H}_0:=\{\pi^l(f)\xi_n \mid f \in 
C^{\infty}_0(G), n \in \N\}$$
is then dense in $L^2(\pol)$.
For $\ph \in C^{\infty}_0(G), n \in \N$, we have, by 
\eqref{E:projkinv} and \eqref{E:pfk}, 
\begin{align*}
\langle \pi^l(\ph)\xi_n, \pi^l(\ph)\xi_n \rangle_{L^2(\pol)}
&=\spl{\pi^l(\ph^**\ph)\xi_n,\xi_n}\\
&=\spl{\pi^l((\ph^**\ph)^{\#})\xi_n,\xi_n}\\
&=\int_{\fa_l^*}{\langle T_
\bsl(\pi^l((\ph^**\ph)^{\#})\xi_n), 
T_\bsl(\xi_n) \rangle_\bsl dp(\bsl)}\\
&=\int_{\fa_l^*}{\langle \pi_\bsl((\ph^**\ph)^{\#})T_
\bsl(\xi_n), 
T_\bsl(\xi_n) \rangle_\bsl dp(\bsl)}\\
&=\int_{\fa_l^*}{\langle \pi_\bsl(\ph^**\ph)T_\bsl(\xi_n), T_
\bsl(\xi_n) \rangle_\bsl dp(\bsl)}\\
&=\int_{\fa_l^*}{\langle \pi_\bsl(\ph)T_\bsl(\xi_n), 
\pi(\ph)T_\bsl(\xi_n) 
\rangle_\bsl dp(\bsl)}\\
&=\int_{\fa_l^*}{\langle T_\bsl(\pi^l(\ph)\xi_n), T_
\bsl(\pi^l(\ph)\xi_n)
\rangle_\bsl dp(\bsl)}.
\end{align*}
Hence, the operator $T$ defined on $\mathscr{H}_0$ by 
$T(\pi^l(\ph)\xi_n)=(\pi_\bsl(\ph)T_\bsl(\xi_n))_\bsl$
extends uniquely to a $G$-equivariant isometric operator
$$T:L^2(\pol) \rightarrow \int_{\fa_l^*}\ch_{\bsl} 
dp(\bsl).$$

It now only remains to prove the surjectivity of $T$.  
Assume therefore that $(\eta_\bsl)_\bsl$ is orthogonal to 
the image of 
$T$. Then for all $\ph$ in 
$L^1(G)$ and $h\in L^1(G)^\#$,
$$\int_{\fa_l^*}{\langle\pi_\bsl(\ph*h)(T\xi)_\bsl,\eta_\bsl
\rangle_\bsl dp(\bsl)}=0,$$
i.e.,
$$\int_{\fa_l^*}{\check{h}(\bsl)\langle\pi_\bsl(\ph)(T\xi)_\bsl,
\eta_\bsl
\rangle_\bsl dp(\bsl)}=0,$$
where $\check{h}(\bsl)$ is the Gelfand transform 
of $h$ restricted to $\fa_l^*$. 
Recall that the set of bounded spherical functions
can be identified with the character space of $L^1(G)^\#$, 
and hence the image of $L^1(G)^\#$ under the Gelfand 
transform separates points in this space. 
It thus follows from the Stone-Weierstrass Theorem that the 
functions $\check{h}$ are dense in the space of continuous 
functions on $\fa_l^*$ that are invariant under the 
action of $W_l$..     
Hence $\langle\pi_\bsl(f)(T\xi)_\bsl,
\eta_\bsl\rangle_\bsl=0$ $p$-almost everywhere. By 
separability of $L^1(G)$, there is a set $U$ with 
$p(\fa_l^*\setminus U)=0$ such that for all $f$ in $L^1(G)$ 
and $\bsl\in U$,
$\langle\pi_\bsl(f)(T\xi)_\bsl,
\eta_\bsl\rangle_\bsl=0$. By cyclicity of $(T\xi)_\bsl$ 
(note that $(T\xi)_\bsl$ is non-zero $p$-almost everywhere), 
$\eta_\bsl$ is zero $p$-almost everywhere.

\end{proof}

\begin{rem} We want to point out that it is actually not 
necessary to prove the 
analytic continuation of $T_\bsn$ (and hence to use the 
theory of Riesz distributions) to derive the decomposition 
of $\pi^l$ (however, the natural operator $T$ above
is then replaced by an abstract one). Indeed, by the 
Cartan-Helgason theorem (\cite[Ch. 
III, Lemma 
3.6]{HEL}) we have
$\ch_\bsl^K=\C\, v_\bsl$ when $\bsl\in\fa_l^*$, and hence
we can set
$$
T_\bsl:\cpol^{K}\ra\ch_\bsl^{K},\quad
f\mt\ti{f}(\bsl)v_\bsl,$$
and by \eqref{E:pfk} we thus obtain an operator 
$T:L^2(\pol)^K\ra\int_{\fa_l^*}{\ch_\bsl^K dp(\bsl)}$.
Assume that we can prove that $T$ intertwines the 
actions of $\cg$. Then for $\ph\in\cg$,
\begin{align*}
\langle \pi^l(\ph)\xi,\xi\rangle&=\langle T\pi^l(\ph)
\xi,T\xi\rangle
=\langle \pi(\ph)
T\xi,T\xi\rangle\\
&=\int_{\fa_l^*}{\langle \pi_\bsl(\ph)
(T\xi)_\bsl,(T\xi)_\bsl\rangle_\bsl}dp(\bsl)\\
&=\int_{\fa_l^*}{\check{\ph}(\bsl)\nm{(T\xi)_\bsl}^2_\bsl}dp(
\bsl),
\end{align*}
where $\check\ph$ is defined by 
$\pi_{\bsl}(\ph)v_\bsl=\check{\ph}(\bsl)v_\bsl$. The proof of 
\cite[Theorem 10]{SEP2} shows that the decomposition of 
$\pi^l$ then follows.
We now prove the intertwining property. It is equivalent to 
the 
equality 
\begin{equation}\label{E:tbp}
\ti{\pi^{l}(\ph)f}(\bsl)=\ti{f}(\bsl)\check{\ph}(\bsl).
\end{equation}
Let $\bsn\in\C^l$. Then for 
$f\in\cpol$ and $\ph\in\cg$ we have, where 
$\bsn'=-(i\bsn+\bsr'_l)$, 
\begin{align*}
\pi_{\bsn}(\ph)&\left(\Delta_{\bsn'}((
\cdot)^*e)\otimes1\right)(g)=\int_{G}{\ph(h)\Dnp(g^*h^{-*}e)
\otimes\Delta^{-\frac{ld}{4}}(he)dh}\\
&=\int_{G}{\ph(h)\Delta^{-\frac{ld}{4}}(he)\left(\int_{K}{
\Dnp(g^*kh^{-*}e)dk}\right)\otimes1dh}\\
&=\int_{G}{\ph(h)\Delta^{-\frac{ld}{4}}(he)\left(\int_{K}{
\Dnp(kh^{-*}e)dk}\right)dh}\left(\Dnp(g^*e)\otimes1\right)\\
&=\int_{G}{\ph(h)\Delta^{-\frac{ld}{4}}(he)\Phi_{\bsn'}(h^{-*}e)dh}
\left(\Dnp(g^*e)\otimes1\right),
\end{align*}
i.e.,
\begin{equation*}
\pi_{\bsn}(\ph)v_\bsn=\check{\ph}(\bsn)v_\bsn,
\end{equation*}
where $\check{\ph}(\bsn)$ is holomorphic on $\C^l$.
If $\Re(-(i\bsn+\bsr'_l))\geq0$, then, by 
Lemma~\ref{L:inter}, 
the operator $T_\bsn$ intertwines the actions of $\cg$,  
and hence
$$\ti{\pi^{l}(\ph)f}(\bsn)=\ti{f}(\bsn)\check{\ph}(\bsn).$$
Thus \eqref{E:tbp} follows by analytic continuation.
\end{rem}

\appendix

\section{Restrictions of distributions}\label{A1}
Let $X$ be a smooth $n$-dimensional manifold. Let $\mathscr{D}(X)$ 
denote the space of compactly supported smooth functions on $X$, 
i.e., the test functions on $X$.
For any chart 
$(V, \phi)$, compact subset $K \subseteq \widetilde{V}:=\phi(V)$, 
and $N \in \mathbb{N}$, consider the seminorm
\begin{equation}
p_{V,K,N}(f):=\sum_{\alpha \in \mathbb{N}^n, |\alpha| \leq N} 
\mbox{sup}_{x \in K} |D^{\alpha}(f \circ \phi^{-1})(x)|  \label{seminorm}
\end{equation}
on the space of smooth functions on $X$. Here 
$D^{\alpha}:=\frac{\partial^{|\alpha|}}{\partial x_1^{\alpha_1} \cdots
\partial x_n^{\alpha_n}}$ for 
$\alpha=(\alpha_1, \ldots, \alpha_n)$, and 
$|\alpha|=|\alpha_1|+\cdots+|\alpha_n|$.  
For a compact set $K \subseteq X$, let $\mathscr{D}(K)$ be
the space of smooth functions with support in $K$ equipped with the 
topology induced by the above seminorms. 
We recall that a distribution on $X$ is a continuous functional on 
$\mathscr{D}(X)=\cup_K \mathscr{D}(K)$ equipped with the inductive 
limit topology. We let $\mathscr{D}'(X)$ denote the space
of distributions on $X$.
Let $\{U_i\}$ be an open cover of $X$. 
Then a linear functional on $\mathscr{D}(X)$ is continuous if and
only if its restriction to every $\mathscr{D}(U_i)$ is continuous.

We will now construct restrictions to a closed submanifold  $Y$ of distributions that have support on $Y$.
To have a well-defined notion of restriction, one can not permit arbitrary extensions to $X$ of test functions on $Y$.
Instead, we will require the extension to be locally constant along some predescribed direction.
This can be made precise using tubular neighbourhoods. 

\begin{defin}
Let $X$ be a smooth $n$-dimensional manifold, and let $Y$ be a 
$k$-dimensional submanifold. A \emph{tubular neighbourhood}
of $Y$ in $X$ consists of a smooth vector bundle $\pi: E \rightarrow Y$, 
an open neighbourhood $Z$ of the image, $\zeta_E(Y)$, of the zero section 
in $E$, and a diffeomorphism $f:Z \rightarrow O$ onto an open 
set $O \subseteq X$ containing $Y$, such that the
diagram
\begin{eqnarray*}
\xymatrix{Z \ar[rd]^{f} &\\
Y \ar[u]^{\zeta_E}  \ar[r]^{j}  & X} 
\end{eqnarray*}
commutes. Here $j:Y \rightarrow X$ is the inclusion map.
\end{defin}

\begin{rem} \label{normalchoice}
Any closed submanifold $Y$ of $X$ admits a tubular neighbourhood (cf. \cite[Ch. IV, F, Thm. 9]{LAN}). Any splitting of the tangent bundle
of $X$ over $Y$, $T(X)\mid_Y=T(Y) \oplus E$, gives such a vector bundle
$E$. In particular, given a Riemannian metric on $X$, $E$ can be 
chosen as the orthogonal complement to $T(Y)$ in $T(X)\mid_Y$.
\end{rem}

Since the concepts we are dealing with are of a local nature
we can without loss of generality assume that $X=Z$ itself is a tubular
neighbourhood of $Y$.

\begin{defin}\label{D:lvc}
A function $f$ on $X$ is said to be \emph{locally vertically constant} around $Y$, l.v.c., if for any $x\in Y$, there exists
an open neighbourhood $W_x$ of $x$ in $X$ such that for $y\in W_x$, $f(y)=f(\pi(y))$. Moreover, if $g$ is a function on $Y$, and $f\mid_Y=g$, 
$f$ is called an l.v.c. extension of $g$.
\end{defin}

\begin{lem}\hfill \label{L:lvcext}
\begin{enumerate}[(i)]
\item Any test function $\ph$ on $Y$ admits a l.v.c extension $\ti\ph\in\mathscr{D}(X)$.
\item An l.v.c. function $f$ that vanishes on $Y$ vanishes on some neighbourhood of $Y$.
\end{enumerate} 
\end{lem}
\begin{proof}
Since $\supp\ph$ is compact, it can be covered by finitely many open neighbourhoods $O_1,\dots,O_N$ diffeomorphic to products $U_i\times V_i\subset 
\R^k\oplus\R^{n-k}$, where $U_i$ is open in $\R^k$ and $V_i$ is an open neighbourhood of
$0$ in $\R^{n-k}$, in such a way that $\pi$ corresponds to the projection onto the first coordinate. Let $(\psi_i)$ be a smooth partition of unity on 
$U=\cup_{i=1}^N\pi(O_i)$ subordinate to the cover $\pi(O_i)$, $i=1,\dots,N$, and let $\ph_i=\ph\psi_i$. Then each $\ph_i$ can be identified with a
test function on $U_i$, and by multiplying this with a test function on $V_i$ which is $1$ on 
some neighbourhood of $0$, we obtain a l.v.c. extension $\ti\ph_i$ of $\ph_i$. Then $\ti\ph=\sum{\ti\ph_i}$ is a l.v.c extension of $\ph$. This proves (i).
For (ii) just observe that if $f$ is an l.v.c function that vanishes on $Y$, then $Y$ is in the complement of the support of $f$.
\end{proof}

\noindent Assume now that $u \in \mathscr{D}'(X)$  has support on the submanifold
$Y$. Then l.v.c. test functions that vanish on a neighbourhood of $Y$ are in the kernel of $u$,
and the preceding lemma enables us to make the following definition.

\begin{defin}\label{D:res}
Let $u \in \mathscr{D}'(X)$ be a distribution with support
on $Y$. The vertical restriction, $u\!\mid_Y$, of $u$ to $Y$ is the unique distribution on $Y$ that satisfies
\begin{equation*}
u\mid_Y(\ph \mid_Y)
=u(\ph), \label{globalrestr}
\end{equation*}
for any $\ph \in \mathscr{D}(X)$ which is l.v.c. around $Y$.
\end{defin}

To see that the functional $u\mid_Y$ really is a distribution on $Y$, it suffices to verify the continuity for test functions
with support in trivialising open sets. In this case the verification is straightforward using the l.v.c extension 
from the proof of Lemma~\ref{L:lvcext}.

\begin{rem}\label{R:choice}
Note that the vertical restriction depends on the choice 
of tubular neighbourhood, i.e., on
a choice of complement $E$ in the splitting of vector
bundles in Remark \ref{normalchoice}. However, when $u$ is a measure on $X$ with support on $Y$, 
the vertical restriction $u \mid _{Y}$ is $u$ itself, now
viewed as a distribution on $Y$.
\end{rem}

We now consider holomorphic families of distributions and
their properties under restriction.

\begin{defin}
Let $\Omega \subseteq\C^m$ be an open set, and let $\{u^z\}_{z \in \Omega}$ be
a family of distributions on the smooth manifold $X$. Then this family 
is called a \emph{holomorphic family of distributions} if the map
$z \mapsto u^z(\ph)$ is holomorphic on $\Omega$ for every
$\ph \in \mathscr{D}(X)$.
\end{defin}

\begin{rem} It follows immediately from Definition~\ref{D:res} that if 
$\{u^z\}_{z \in \Omega}$ is a holomorphic family of
distributions with support on $Y$, then the family $\{u^z \mid_Y\}_{z \in \Omega}$
is a holomorphic family of distributions on $Y$.
\end{rem}

\begin{pro}\label{P:invres} Let $\Omega\subseteq\C^n$ be open and connected, and
let $\{u^z\}_{z \in \Omega}$ be a holomorphic family of
distributions on $X$ with support on $Y$. Assume that
there exists an open subset $U \subseteq \Omega$, such that
$u^z$ is a measure with support on $Y$ for $z \in U$.
Then the whole family
$\{u^z\mid_Y\}_{z \in \Omega}$ is independent of
$E$.
\end{pro}

\bibliographystyle{amsalpha}
\bibliography{ourref}

\end{document}